\documentclass[12pt]{amsart}
\usepackage{amscd,amsmath,amsthm,amssymb,graphics}
\usepackage{lmodern,pst-node}
\usepackage{multicol}
\usepackage{epic,eepic}
\usepackage{amsfonts,amssymb,amscd,amsmath,enumerate,verbatim}

\psset{unit=0.7cm,linewidth=0.8pt,arrowsize=2.5pt 4}

\newpsstyle{fatline}{linewidth=1.5pt}
\newpsstyle{fyp}{fillstyle=solid,fillcolor=verylight}
\definecolor{verylight}{gray}{0.97}
\definecolor{light}{gray}{0.9}
\definecolor{medium}{gray}{0.85}
\definecolor{dark}{gray}{0.6}

\def\NZQ{\mathbb}               

\def\ZZ{{\NZQ Z}}

\def\frk{\mathfrak}               

\def\pp{{\frk p}}

\def\mm{{\frk m}}

\def\nn{{\frk n}}
\def\Phi{{\frk N}}
\def\opn#1#2{\def#1{\operatorname{#2}}} 
\opn\chara{char} \opn\length{\ell} \opn\pd{pd} \opn\rk{rk}
\opn\projdim{proj\,dim} \opn\injdim{inj\,dim} \opn\rank{rank}
\opn\depth{depth} \opn\grade{grade} \opn\height{height}
\opn\embdim{emb\,dim} \opn\codim{codim}

\opn\Tr{Tr} \opn\bigrank{big\,rank}
\opn\superheight{superheight}\opn\lcm{lcm}
\opn\trdeg{tr\,deg}
\opn\reg{reg} \opn\lreg{lreg} \opn\ini{in} \opn\lpd{lpd}
\opn\size{size}\opn\bigsize{bigsize}
\opn\cosize{cosize}\opn\bigcosize{bigcosize}
\opn\sdepth{sdepth}\opn\sreg{sreg}
\opn\link{link}\opn\fdepth{fdepth}
\opn\div{div} \opn\Div{Div} \opn\cl{cl} \opn\Cl{Cl}
\opn\Spec{Spec} \opn\Supp{Supp} \opn\supp{supp} \opn\Sing{Sing}
\opn\Ass{Ass} \opn\Min{Min}\opn\Mon{Mon} \opn\dstab{dstab} \opn\astab{astab}
\opn\Syz{Syz}
\opn\Ann{Ann} \opn\Rad{Rad} \opn\Soc{Soc}
\opn\Im{Im} \opn\Ker{Ker} \opn\Coker{Coker} \opn\Am{Am}
\opn\Hom{Hom} \opn\Tor{Tor} \opn\Ext{Ext} \opn\End{End}
\opn\Aut{Aut} \opn\id{id}

\opn\nat{nat}
\opn\pff{pf}
\opn\Pf{Pf} \opn\GL{GL} \opn\SL{SL} \opn\mod{mod} \opn\ord{ord}
\opn\Gin{Gin} \opn\Hilb{Hilb}\opn\sort{sort}
\opn\aff{aff} \opn\con{conv} \opn\relint{relint} \opn\st{st}
\opn\lk{lk} \opn\cn{cn} \opn\core{core} \opn\vol{vol}
\opn\link{link} \opn\star{star}\opn\lex{lex}
\opn\gr{gr}

\def\pot#1#2{#1[\kern-0.28ex[#2]\kern-0.28ex]}

\opn\dirlim{\underrightarrow{\lim}}
\opn\inivlim{\underleftarrow{\lim}}
\let\dirsum=\oplus

\let\iso=\cong

\let\Dirsum=\bigoplus

\let\to=\rightarrow

\def\Implies{\ifmmode\Longrightarrow \else
        \unskip${}\Longrightarrow{}$\ignorespaces\fi}
\def\implies{\ifmmode\Rightarrow \else
        \unskip${}\Rightarrow{}$\ignorespaces\fi}
\def\iff{\ifmmode\Longleftrightarrow \else
        \unskip${}\Longleftrightarrow{}$\ignorespaces\fi}

\let\:=\colon
\newtheorem{Theorem}{Theorem}[section]
 
 \newtheorem{Corollary}[Theorem]{Corollary}
 \newtheorem{Proposition}[Theorem]{Proposition}
 \newtheorem{Remark}[Theorem]{Remark}
 
 \newtheorem{Example}[Theorem]{Example}
 
 \newtheorem{Definition}[Theorem]{Definition}

\let\epsilon\varepsilon
\let\kappa=\varkappa
\def\qed{\ifhmode\textqed\fi
      \ifmmode\ifinner\quad\qedsymbol\else\dispqed\fi\fi}
\def\textqed{\unskip\nobreak\penalty50
       \hskip2em\hbox{}\nobreak\hfil\qedsymbol
       \parfillskip=0pt \finalhyphendemerits=0}
\def\dispqed{\rlap{\qquad\qedsymbol}}

\opn\dis{dis}
\def\pnt{{\raise0.5mm\hbox{\large\bf.}}}

\opn\Lex{Lex}

\begin{document}
 \title {Anticanonical modules of Segre products}

\author {Viviana Ene, J\"urgen Herzog, and Dumitru I.\ Stamate}

\address{Viviana Ene, Faculty of Mathematics and Computer Science, Ovidius University, Bd.\ Mamaia 124,
 900527 Constanta, Romania} \email{vivian@univ-ovidius.ro}

\address{J\"urgen Herzog, Fachbereich Mathematik, Universit\"at Duisburg-Essen, Campus Essen, 45117
Essen, Germany} \email{juergen.herzog@uni-essen.de}

\address{Dumitru I. Stamate, ICUB/Faculty of Mathematics and Computer Science, University of Bucharest, Str. Academiei 14, Bucharest -- 010014, Romania}
\email{dumitru.stamate@fmi.unibuc.ro}

\dedicatory{Dedicated to Dorin Popescu at his $70$th birthday}

\thanks{}

\begin{abstract}
We compute  the anticanonical module of the Segre product of standard graded $K$-algebras and determine its depth.
 \end{abstract}

 \maketitle

\section*{Introduction}

Let $R$ be a standard graded Cohen--Macaulay $K$-algebra with canonical module $\omega_R$.
If $R$ is Gorenstein, then $\omega_R \cong R$. Otherwise,
if $R$ happens to be generically Gorenstein, then $\omega_R$ may be identified with an ideal of height one.
In this case one may study the powers $\omega_R^a$ with $a\in \ZZ$, and  ask for which integers $a$ the ideal $\omega_R^a$ is  Cohen--Macaulay. This question has been answered in \cite{BRW} for the canonical module of determinantal rings.

The ideal $\omega_R^{-1}$ is commonly called the anticanonical ideal. Correspondingly, for any Cohen--Macaulay $K$-algebra $R$,  one calls  the $R$-dual of $\omega_R$, namely $\omega_R^*=\Hom_R(\omega_R,R)$, the anticanonical module of $R$. This module is of particular interest.
Indeed, assume as before that $R$ is generically Gorenstein.
Let $S$ be  a Noether normalization  of $R$ and $\sigma: Q(R)\to Q(S)$ a trace map, where
$Q(S)$ and $Q(R)$ denote the total ring of fractions of $S$ and $R$, respectively. Then   the complementary module $\mathfrak{C}^\sigma_{R/S}$ is defined and it is shown in \cite[Satz 7.20]{HK} that $\mathfrak{C}^\sigma_{R/S}$ is isomorphic to the cano\-nical ideal of $R$. Its inverse is the Dedekind different $\mathfrak{D}^\sigma(R/S)$. In number theory this invariant  encodes the ramification  of the corresponding field extension, and by definition it is isomorphic to the anticanonical ideal of $R$.

The anticanonical ideal appears as well  in other homological contexts. As observed in \cite[Lemma 2.1]{HSt}, $\omega_R^{-1}\omega_R$ describes the non-Gorenstein locus of $R$.  Thus, if $R$ is Gorenstein on the punctured spectrum of $R$, then  $\height (\omega_R^{-1}\omega_R)=\dim R$. On the other hand, the anticanonical module  is isomorphic to an ideal of height one. It  may be a Cohen--Macaulay ideal or not, while $\omega_R$ is always Cohen--Macaulay.

In this paper we study the anticanonical module of the Segre product of standard graded Cohen--Macaulay $K$-algebras.
One motivation for us to study this question came from the study of Hibi rings \cite{Hibi}.
Their canonical ideals are well understood,  while at present this is not the case for their anticanonical ideals.
The Hibi ring of a sum of finite posets is just the Segre product of the Hibi rings of each  summand.
Thus if we understand how to compute the anticanonical module of Segre products, then for the study of the  anticanonical ideal of a Hibi ring one may restrict oneself to the case that the underlying poset is not a sum of two posets.

Our study of the canonical and anticanonical module of Segre products is based on results from the fundamental
 paper ``On graded rings, I'' by Goto and Watanabe \cite{GW}. There it is shown \cite[Theorem 4.3.1]{GW} that if $R$ and $S$ are  two
standard graded algebras with $\dim R, \dim S\geq 2$ and $T=R\sharp S$, then $\omega_T=\omega_R\sharp\omega_S$.
One would then expect that $\omega_T^\ast \cong \omega_R^\ast \sharp \omega_S^\ast$, and of course this would be true if,  more generally,   for any finitely generated graded $R$-module $M$  and any finitely generated graded $S$-module $N$, it would follow that $(M\sharp N)^\ast\iso M^\ast\sharp N^\ast$.
Example \ref{ex:notfriendly} shows that this is not always the case.
Therefore,  in the first section of this paper we introduce  friendly families of  algebras.
We say that the family of standard graded $K$-algebras $R_1, \dots, R_m$ is friendly if the natural map
$$
\alpha:R_1(a_1)^\ast \sharp\cdots \sharp R_m(a_m)^\ast \to (R_1(a_1)\sharp\cdots \sharp R_m(a_m))^\ast
$$
is an isomorphism for all integers $a_1,\dots, a_m$. When $m=2$ we say that $(R_1, R_2)$ is a friendly pair.
We show in  Theorem \ref{commute}, that if $R_1, \dots, R_m$ is a friendly family, then $(M_1\sharp \cdots \sharp M_m)^\ast\iso M_1^\ast\sharp \cdots \sharp M_m^\ast$ when
$M_i$ is a finitely generated graded $R_i$-module, $i=1, \dots, m$.
We also prove in this section that if $R_1, \dots, R_m$ are standard graded toric rings with
 $\depth R_i \geq 2$  for $i=1, \dots, m,$  then $R_1, \dots, R_m$ is a friendly family, see Theorem~\ref{ToricSegre}.

We apply these results in Section \ref{sec:anticanonical} to conclude that for any pair $(R,S)$ of standard graded Cohen--Macaulay toric rings with  $\dim R,\dim S\geq 2$,
we have the desired isomorphism  $\omega_T^\ast \cong \omega_R^\ast \sharp \omega_S^\ast$, where  $T=R\sharp S$.
This follows from Corollary~\ref{anticanonic}, where this result is formulated more generally for friendly pairs.
In particular, if we assume that $R$ and $S$ are Gorenstein rings of dimension $\geq 2$ and with negative $a$-invariant,
we deduce from this result  in Proposition~\ref{reflexiv} that the canonical module of the Segre product $T=R\sharp S$ is reflexive,
 which in turn,  as a consequence of \cite[Theorem 7.31]{HK},  implies  that the localization $T_P$  is Gorenstein for any height $1$ prime ideal $P$ of $T.$

The concept of a friendly family can be extended to positively graded or multigraded algebras. Results similar to Theorems \ref{commute} and \ref{ToricSegre} may be formulated in that generality. In view of later applications in Section \ref{sec:anticanonical} we restricted to the standard graded case.

The next part of Section \ref{sec:anticanonical} is devoted to determine the depth of the
anticanonical module for $R\sharp S$ when $(R,S)$ is a friendly pair of standard graded Cohen--Macaulay $K$-algebras.
Inspired by Proposition~4.2.2 in \cite{GW}, we determine,  in terms of $a$ and $b$,  the depth of $R(a)\sharp S(b)$ when $R$ and $S$ are standard graded Gorenstein $K$-algebras,
 and use this result in Corollary~\ref{CM} to characterize those pairs of friendly Gorenstein $K$-algebras for which the anticanonical module of the Segre product $R\sharp S$ is Cohen--Macaulay.

In Theorem~\ref{mfactors-2}  we generalize Corollary~\ref{CM} to Segre products of  finitely many Gorenstein algebras.
Namely,  given $R_1, \dots, R_m$ a friendly 	family of standard graded Gorenstein  $K$-algebras of dimension at least two, with $-\rho_i$ denoting the $a$--invariant of $R_i$ for $i=1,\dots, m$,
Theorem \ref{mfactors-2} describes for which integers $a$   the module
$$
T^{\{a\}}=R_1(-a\rho_1)\sharp\cdots \sharp R_m(-a\rho_m)$$
of uniform twists of the canonical module of $T=\sharp_{i=1}^m R_i$
 is a Cohen-Macaulay $T$-module.
In particular, assuming $\rho_1\geq \cdots \geq \rho_m$,  the anticanonical module $\omega_T^\ast$ is Cohen-Macaulay if and only if
\[2^{m-1}\rho_m>2^{m-2}\rho_{m-1}>\cdots> 2\rho_2>\rho_1,\]
	see Corollary \ref{cor:cm-chain}.

In Proposition \ref{prop:cm-twists} we see that if all $\rho_i$'s are positive and   not all equal (i.e. $T$ is Cohen-Macaulay, and not a Gorenstein ring),
then the set of $a$'s such that $T^{\{a\}}$ is Cohen-Macaulay represents a bounded interval which is explicitly described.

If we assume, moreover, that $\otimes_{i=1}^m R_i$ is a domain, which is for instance the case when the $R_i$'s are toric rings, then the canonical module and its dual
 may be identified with ideals in $T$.
In Proposition \ref{prop:powers-cm}, which is the last result in this paper, we describe  which powers  $\omega_T^a$ are Cohen-Macaulay, using the previous Proposition~\ref{prop:cm-twists}.


\section{Friendly families of standard graded algebras}
\label{sec:friendly}

Let $R$ be a standard graded $K$--algebra, where $K$ is a field, and $M$ a finitely generated graded $R$--module. Then the dual of $M,$
$M^\ast=\Hom_R(M,R)$ is graded by $M^\ast=\dirsum_{i\in \ZZ} \Hom_R^i (M,R)$ where
\[\Hom_R^i(M,R)=\{\varphi\in \Hom_R(M,R): \varphi(M_k)\subseteq R_{k+i} \text{ for all }k\}.\]
In what follows, we will  also consider the graded $K$--dual of $M,$ namely \[M^\vee=\Hom_K(M,K)=\dirsum_{i\in \ZZ}\Hom_K(M_{-i},K).\]

Let $S$ be another standard graded $K$--algebra and $N$ a finitely generated graded $S$--module.
Let $T=R\sharp S$ be the Segre product of $R$ and $S.$
The algebra $T$ is also standard graded with the grading $T=\dirsum_{i\geq 0} (R_i\otimes_K S_i).$ The Segre product $M\sharp N$ is a graded $T$--module with
$M\sharp N=\dirsum_{j\in \ZZ} (M_j \otimes_K N_j).$

There is a natural homogeneous $T$--module homomorphism
$\alpha: M^\ast \sharp N^\ast \to (M\sharp N)^\ast$ defined as follows. An element of $(M^\ast \sharp N^\ast)_i =M_i^\ast \otimes N_i^\ast$
is    a sum of elements of the form $\varphi\otimes \psi$ with $\varphi \in M_i^\ast, \psi\in N_i^\ast.$
Here $M^\ast_i$ denotes
the $i$-th homogeneous component of $M^\ast.$ Then
$\alpha(\varphi \otimes \psi)\in (M\sharp N)^\ast_i$ is the map $(M\sharp N)_k\to R_{k+i}\otimes S_{k+i}$ given by
$m\otimes n \mapsto \varphi(m)\otimes \psi(n).$

We can iterate this construction. In general, given $R_1, \dots, R_m$ standard graded algebras, and   $M_i$ a graded  $R_i$-module, $i=1, \dots, m$, we can construct a natural map
$$
\alpha:M_1^\ast \sharp \cdots \sharp M_m^\ast \to (M_1\sharp\cdots \sharp M_m)^\ast.
$$

\begin{Definition}
\label{def:friendly-family}
{\em
The family of algebras $R_1,\dots, R_m$ is called   {\em friendly} if the natural map
$$
\alpha:R_1(a_1)^\ast \sharp \cdots \sharp R_m(a_m)^\ast \to (R_1(a_1)\sharp\cdots \sharp R_m(a_m))^\ast
$$
is an isomorphism for all integers $a_1, \dots, a_m$.

In case $m=2$ we say that $(R_1, R_2)$ is a friendly pair.
}
\end{Definition}

Not all families of standard graded $K$--algebras are  friendly, as the following example shows.

\begin{Example}
\label{ex:notfriendly}
{\em
%
Let $x$ and $y$ be indeterminates over $K$ and let $R=K[x]/(x^3)$ and $S=K[y]/(y^2)$. 
Then $R\sharp S= K \dirsum Kxy$. One has $R(2)\sharp S(1) = Kx \dirsum Kx^2y$, with $x$ sitting  in degree $-1$, 
therefore $R(2)\sharp S(1)\cong  (R\sharp S)(1)$. On the other hand, $R(2)^* \sharp S(1)^* \cong R(-2)\sharp S(-1) =Ky$.

Thus $(R(2)\sharp S(1))^*\cong (R\sharp S)(-1)$ has two nonzero components, while $R(2)^* \sharp S(1)^*$ has only one nonzero component. Hence these two modules cannot be isomorphic and $(R,S)$ is not a friendly pair.
}
\end{Example}

\medskip
However, we show in Theorem \ref{ToricSegre} that any finite collection of standard graded toric rings is friendly.
The following result will be useful to that purpose. At the same time, Theorem \ref{commute} gives a wider class of modules for which the map $\alpha$ used before is an isomorphism.

\begin{Theorem}
\label{commute}
Let $R_1, \dots, R_m$ be a friendly family of standard graded $K$-algebras, and $T=\sharp_{i=1}^m R_i$.
If $M_i$ is a finitely generated graded $R_i$-module for $i=1, \dots, m$, then the natural map
$\alpha: \sharp_{i=1}^m M_i^{\ast} \to (\sharp_{i=1}^m M_i)^\ast$ is an isomorphism of graded $T$-modules.
\end{Theorem}

\begin{proof}
The statement is proved by induction on $$r=|\{i:  M_i \text { is not free}, i=1, \dots, m \}|.$$

If $r=0$, by the definition of a friendly family and taking into account that the dual functor and the Segre product commute with finite direct sums,
it follows that $\alpha$ is an isomorphism.

Let $r>0$ and assume that the conclusion of the theorem holds when at most $r-1$ of the $M_i$'s are not free.
Without loss of generality we may assume $M_1$ is not free. Let
\begin{equation}
\label{eq:presentation}
G\to F\to M_1 \to 0
\end{equation}
be a presentation of $M_1$ by finitely generated graded free $R_1$-modules.  We dualize \eqref{eq:presentation} and then apply the exact functor $-\sharp(\sharp_{i>1}M_i)$ to obtain the
exact sequence
\begin{equation}
\label{eq:pres1}
0\to \sharp_{i=1}^m M_i^\ast \to F^\ast \sharp (\sharp_{i>1}M_i) \to G^\ast \sharp (\sharp_{i>1}M_i).
\end{equation}

Similarly, if we first apply $-\sharp(\sharp_{i>1}M_i)$ to \eqref{eq:presentation} and then we dualize, we have the exact sequence
\begin{equation}
\label{eq:pres2}
0\to (\sharp_{i=1}^m M_i)^\ast \to (F\sharp (\sharp_{i>1}M_i))^\ast \to (G\sharp (\sharp_{i>1}M_i))^\ast.
\end{equation}

The maps $\alpha$ induce  a  map between the chains \eqref{eq:pres1} and \eqref{eq:pres2}. By the inductive hypothesis, the two rightmost maps $\alpha$ are
isomorphisms, hence using the $5$-Lemma (\cite[Exercise 1.3.3]{Weibel}) we get 
that the leftmost map $\alpha$ is an isomorphism, as well.
This finishes the proof.
\end{proof}

Let us make a simple remark before proceeding to the main result of this section.

\begin{Remark}\label{shiftS}{\em
Let $R$ and $S$ be standard graded algebras such that  the natural map
\[
\alpha: R(a)^\ast \sharp S \to (R(a)\sharp S)^\ast  
\]
  is an isomorphism  for all $ a\in \ZZ$.
Then
\begin{eqnarray*}
(R(a)\sharp S(b))^\ast &=& ((R(a-b)\sharp S)(b))^\ast= (R(a-b)\sharp S)^\ast (-b) \\
&\cong& ((R(a-b))^\ast \sharp S)(-b)= (R(b-a)\sharp S)(-b) \\
&=& R(-a)\sharp S(-b) \\
&=& R(a)^\ast \sharp S(b)^\ast \text{ \quad for all } a,b\in \ZZ.
\end{eqnarray*}
Thus  $(R,S)$ is a friendly pair in the sense of Definition \ref{def:friendly-family}.
}
\end{Remark}

We note that, by \cite[Corollary 16]{KR},
the Segre product of two standard  graded toric rings $R, S$ is a standard graded toric ring, as well.
For the convenience of the reader we include a short proof of this statement.

\begin{Proposition}\label{Toric}
Tensor products and Segre products of standard graded toric rings are standard graded toric rings, too. In particular, they are domains.
\end{Proposition}

\begin{proof}
Let $R=K[t^{a_1},\ldots,t^{a_n}]=K[A]$ where $A$ is an integer matrix with column vectors $a_1,\ldots,a_n$ and
$S=K[s^{b_1},\ldots,s^{b_m}]=K[B]$ where $B$ is an integer matrix with column vectors $b_1,\ldots,b_m.$ Let $I$ be the kernel of the
morphism $K[x_1,\ldots,x_n]\to K[A]$ induced by $x_i\mapsto t^{a_i}$ for $1\leq i\leq n,$ and $J$ the kernel of the morphism
$K[y_1,\ldots,y_m]\to K[B]$ induced by $y_j\mapsto s^{b_j}$ for $1\leq j\leq m.$ Then
\[
R\otimes _K S\cong K[x,y]/IK[x,y]+JK[x,y]
\] where $K[x,y]$ denotes the polynomial ring in the variables $x_1,\ldots,x_n,y_1,\ldots,y_m.$

Let $L_1,L_2$ be the following lattices in $\ZZ^n$: $L_1=\{c\in \ZZ^n: Ac=0\}$ and $L_2=\{d\in \ZZ^m: Bd=0.\}$ Then
$I=I_{L_1}$ and $J=I_{L_2}$ where $I_{L_j}$ is the lattice ideal of $L_j$ for $j=1,2. $

Let \[L=\left\{\left(
\begin{array}{l}
	c\\
	d
\end{array}
\right)\in \ZZ^n\dirsum \ZZ^m: \left(
\begin{array}{cc}
	A & 0\\
	0 & B
\end{array}
\right)\left(
\begin{array}{l}
	c\\
	d
\end{array}
\right)=0\right\}.\]
Then $I_L=(I_{L_1}+I_{L_2})K[x,y].$ As $\ZZ^n/L_1$ and $\ZZ^m/L_2$ are torsion free,  $\ZZ^n\dirsum \ZZ^m/I_L$ is torsion free as well, hence $I_L$ is a prime ideal (\cite[Theorem~8.2.2]{Vi} or \cite[Theorem~2.1]{Ei}) and
\[R\otimes_K S\cong K[x,y]/I_L=K[C],
\]
where $C=\left(
\begin{array}{cc}
	A & 0\\
	0 & B
\end{array}
\right).$

This implies that the morphism $R\otimes_K S\to K[t^{a_1},\ldots,t^{a_n},s^{b_1},\ldots,s^{b_m}]=K[C]$ induced by
$t^{a_i}\otimes 1\mapsto t^{a_i}$ and $1\otimes s^{b_j}\mapsto s^{b_j}$ is an isomorphism. Since $R$ and $S$ are standard graded, we get
\[
R\sharp S =K[t^{a_i}s^{b_j}:1\leq i\leq n, 1\leq j\leq m].
\]
\end{proof}

\begin{Theorem}\label{ToricSegre}
Any finite family of standard graded toric rings of   depth at least two is friendly.
\end{Theorem}

\begin{proof}
For the proof we make a series of reductions.
Firstly, we note that it is enough to prove the statement for friendly pairs, and then argue by induction on the cardinality of the family of algebras.
Indeed, if $R_1, \dots, R_m$ are a friendly family of standard graded toric rings of   depth  at least 2, and  $m\geq 3$, then
for all $a_1, \dots, a_m \in \ZZ$ we have
\begin{eqnarray}
\nonumber (\sharp_{i=1}^m R_i(a_i))^\ast  &\cong& ((\sharp_{i=1}^{m-1} R_i(a_i))\sharp R_m(a_m))^\ast \\
\label{eq:oneiso} &\cong&  (\sharp_{i=1}^{m-1} R_i(a_i))^\ast \sharp R_m(a_m)^\ast \\
\nonumber &\cong&  (\sharp_{i=1}^{m-1} R_i(a_i)^\ast ) \sharp R_m(a_m)^\ast  \\
\nonumber &\cong&  \sharp_{i=1}^{m} R_i(a_i)^\ast.
\end{eqnarray}
  Here, for the isomorphism in \eqref{eq:oneiso} we used Theorem \ref{commute} applied to the algebras
 $\sharp_{i=1}^{m-1}R_i$ and $R_m$ that are toric (cf.  Proposition \ref{Toric}) and form a 
friendly  pair, see the base case which is discussed in the next paragraph.

For the induction step to work we need to prove the base case $m=2$ in a slightly more general setup.
We let $R$ and $S$ be toric standard graded algebras with $\depth R \geq 2$ or $\depth S \geq 2$.
Without loss of generality, we may assume that $\depth S\geq 2$.
According to Remark \ref{shiftS}, for checking that $(R,S)$  is a friendly pair it suffices to show that
 $(R(a)\sharp S)^\ast \cong R(a)^\ast \sharp S^\ast$ or, equivalently, $(R(a)\sharp S)^\ast \cong R(-a)\sharp S$ for all $a\in \ZZ.$
Suppose that we have proved the latter isomorphism for $a>0,$ and let $a<0.$ Then
\[
(R(a)\sharp S)^\ast = ((R\sharp S(-a))(a))^\ast =(R\sharp S(-a))^\ast (-a)\cong (R\sharp S(a))(-a)=R(-a)\sharp S.
\]

Thus, in order to complete the proof, we need to show that if $a>0,$ then \[(R(-a)\sharp S)^\ast \cong R(a)\sharp S. \]

As a $T=R\sharp S$-- module, $R(-a)\sharp S$ is generated by $1\otimes S_a.$ Let $f$ be a monomial in  $R_a$.
Since $R$ is a domain, the multiplication map
$R(-a)\stackrel{\cdot f}{\rightarrow}R$ is injective, and since $\sharp S$ is an exact functor, the induced map $R(-a)\sharp S \to R\sharp
S$ is injective. Thus, the map $R(-a)\sharp S \to R\otimes _K S$ is injective as well. It follows that $R(-a)\sharp S$ is isomorphic to
the ideal $J$ of $T$ generated by $f\otimes S_a.$ Therefore, up to a shift, $(R(-a)\sharp S)^\ast$ is isomorphic to the inverse $J^{-1}$ of $
J.$

Since $J$ is a monomial ideal, it follows that $J^{-1}$ is a fractionary monomial ideal.
Let $x=\frac{g_1\otimes h_1}{g_2\otimes h_2}=\frac{g_1}{g_2}\otimes \frac{h_1}{h_2}\in J^{-1}$ where $g_1,g_2,h_1,h_2$ are monomials
with $\deg g_1=\deg h_1=i$ and $\deg g_2=\deg h_2=j.$ Then $f \frac{g_1}{g_2}\otimes \frac{h_1}{h_2}S_a\subseteq R\sharp S.$
 This implies that  $h_1/h_2\in U^{-1}$, where $U$ denotes the ideal of $S$ generated by its homogeneous
component of degree $a$.  Since   $\depth S\geq 2, $ 
and $U$ is  a primary ideal in $S$ with radical the maximal graded ideal of $S$,  we have $\grade U \geq 2,$
 which implies that   $U^{-1}\cong S$   by \cite[Exercise 1.2.24]{BHbook}. 
Hence $h_1/h_2\in S$ which means $i-j\geq 0.$ Then $f g_1/g_2\in R_{a+i-j}, $ that is, $g_1/g_2\in f^{-1}R_{a+i-j}.$
We obtain $xf\in (R(a)\sharp S)$. Therefore, we have proved that $J^{-1}\subseteq f^{-1}(R(a)\sharp S).$

On the other hand, $J=f(R(-a)\sharp S)$ and $f^{-1}(R(a)\sharp S) f (R(-a)\sharp S)\subseteq R\sharp S,$ which implies that
$f^{-1}(R(a)\sharp S)\subseteq J^{-1}.$ Hence $J^{-1}=f^{-1}(R(a)\sharp S),$ which proves our claim.
\end{proof}


\section{The anticanonical module  and its twists}
\label{sec:anticanonical}

Let $R$ be a standard graded  $K$--algebra of $\dim R=r$ with maximal graded ideal $\mm.$ The module
$\omega_R=(H_{\mm}^r(R))^\vee\cong \Hom_K(H_{\mm}^r(R), K)$ is called the canonical module of $R.$ Here $H_{\mm}^r(R)$ denotes the
$r$-th local cohomology module of $R.$ We refer the reader to \cite{GW} for several properties of $\omega_R.$

Let $\omega_R^{\ast}$ be the dual of $\omega_R,$ that is, $\omega_R^\ast=\Hom_R(\omega_R,R).$ The module $\omega_R^{\ast}$ is called the anticanonical module of $R.$

 When $R$ is a Cohen-Macaulay normal domain with canonical ideal $\omega_R$, it is customary to define the anticanonical
module of $R$ as the element in the divisor class group $\Cl(R)$ which is the inverse to the canonical class $[\omega_R]$.
Then $-[\omega_R]=[\omega_R^{-1}]$ by the way $\Cl(R)$ is defined.
By \cite[Corollary 3.3.19]{BHbook} one has that  $\omega_R$ is a divisorial ideal, i.e. $(\omega_R^{-1})^{-1}=\omega_R$. It follows that $\omega_R^{-1}=((\omega_R^{-1})^{-1})^{-1}$. On the other hand, $((\omega_R^{-1})^{-1})^{-1}$ is a divisorial ideal of $R$, see \cite[Lemma 4.48(c)]{BG}.
This implies that $\omega_R^{-1}$  is a divisorial ideal, as well. We have $\omega_R^{-1}\cong \omega_R^\ast$, since  $\grade \omega_R >0$,  see \cite[Lemma 3.14]{Lindo}. This shows that $-[\omega_R]=[\omega_R^\ast]$ and our definition for the anticanonical module of $R$ agrees with the classical one when $R$ is normal.

\medskip
Theorem~\ref{commute} has  several consequences.

\begin{Corollary}\label{anticanonic}
Let $(R,S)$ be a friendly pair of  standard graded algebras with  $\depth R\geq 2$, $\depth S\geq 2$  and $T=R\sharp S.$ Then
\[
\omega_T^\ast \cong \omega_R^\ast \sharp \omega_S^\ast.
\] In particular, if $S$ is Gorenstein with $a$--invariant $\sigma,$ then
\[
\omega_T^\ast\cong \omega_R^\ast \sharp S(-\sigma).\]
Moreover, if $\omega_R^\ast$ is generated by $w_1,\ldots,w_m$ with $\deg w_i=\rho_i$
for $1\leq i\leq m$ and $\sigma \leq\min_{1\leq i\leq m} \rho_i$, then $\omega_T^\ast$ is generated by $w_i\otimes S_{\rho_i-\sigma},1\leq i\leq m$.
\end{Corollary}

\begin{proof}
By \cite[Theorem 4.3.1]{GW}, we have $\omega_T\cong \omega_R\sharp\omega_S.$ Hence, the desired isomorphism follows by applying Theorem~\ref{commute}.
The second isomorphism follows since $\omega_S=S(\sigma)$ if $S$ is Gorenstein with $a$--invariant $\sigma.$
\end{proof}

\begin{Remark}
\label{rem:anti-noniso}
{\em
With notation as in Corollary \ref{anticanonic}, it is not true in general that $\omega_T^\ast$ and $\omega_R^\ast \sharp \omega_S^\ast$ are isomorphic. Indeed, for the algebras $R=K[x]/(x^3)$ and $S=K[y]/(y^2)$ in Example \ref{ex:notfriendly}
one has $\omega_R\cong R(2)$, $\omega_S\cong S(1)$, hence $\omega_R^\ast \sharp \omega_S^\ast = R(-2)\sharp S(-1)=Ky$.
We note that the ring $R\sharp S$ is graded isomorphic with the ring $S$, hence $\omega_T^*$ has two nonzero components. This shows that  $\omega_T^\ast$ and $\omega_R^\ast \sharp \omega_S^\ast$ are not isomorphic.
}
\end{Remark}

\begin{Corollary}\label{Gor}
Let $(R,S)$ be a pair of friendly standard graded algebras and $T=R\sharp S.$ Assume that $R,S$ are Gorenstein of $a$-invariants $\rho,$ respectively, $\sigma$, and $\dim R,$ $ \dim S\geq 2.$ Then $\omega_T^\ast \cong R(-\rho)\sharp S(-\sigma),$ hence $\omega_T^\ast$ is isomorphic to $R(\sigma)\sharp S(\rho)$ up to a shift.
\end{Corollary}

\begin{proof}
 Under our hypothesis, we have $\omega_R=R(\rho)$ and $\omega_S=S(\sigma).$ It follows that $\omega_T\cong R(\rho)\sharp S(\sigma)$, thus $\omega_T^\ast \cong R(-\rho)\sharp S(-\sigma)=  (R(\sigma)\sharp S(\rho))(-\rho-\sigma)$.
\end{proof}

\begin{Proposition}\label{reflexiv}
Let $(R,S)$ be a pair of friendly standard graded algebras and $T=R\sharp S.$ Assume that $R,S$ are Gorenstein of $a$-invariants $\rho<0,$ respectively, $\sigma<0$, and $\dim R,$ $ \dim S\geq 2.$
Then
\begin{itemize}
	\item [(i)] the canonical module $\omega_T$ is reflexive;
	\item [(ii)] the localization $T_P$  is Gorenstein for any height $1$ prime ideal $P$ of $T.$
\end{itemize}
\end{Proposition}

\begin{proof}
(i). We have $(R(\rho)\sharp S(\sigma))^{\ast\ast}\cong (R(-\rho)\sharp S(-\sigma))^\ast\cong R(\rho)\sharp S(\sigma),$ thus the canonical module of $T$ is reflexive.

Statement (ii) follows by \cite[Theorem 7.31]{HK}.
\end{proof}

\begin{Remark} {\em Note that, in the setting of  Corollary~\ref{Gor}, $\omega_T$  is generated by $1\otimes S_{\sigma-\rho}$ if $\rho<\sigma$ and $\omega_T^\ast$  is generated by $R_{\sigma-\rho}\otimes 1.$ This shows that one can easily find examples of rings $R,S$ such that $\mu(\omega_T)>\mu(\omega_T^\ast)$ or $\mu(\omega_T^\ast)>\mu(\omega_T).$ By $\mu(M)$ we denote the minimal number of homogeneous generators of a finitely generated graded module $M.$
}
\end{Remark}

It would be interesting to see under which conditions the anticanonical module of Corollary~\ref{Gor} is Cohen-Macaulay. To this aim, we
first prove  a slightly more general result, inspired by \cite[Proposition 4.2.2]{GW}.

\begin{Proposition}\label{Depth}
Let $R,S$ be Gorenstein standard graded algebras with $\dim R=r\geq 1$ and $\dim S=s\geq 1.$ Let $\rho$ be the $a$--invariant of $R$ and
$\sigma$ the $a$--invariant of $S.$ Let $M=R(a)\sharp S(b)$ for some integers $a,b.$ The following statements hold:
\begin{itemize}
	\item [(i)] If $r=s>1,$ then $\depth M=r$ if $a-b\geq-\sigma$ or $a-b\leq \rho,$ and $M$ is Cohen-Macaulay if $-\sigma > a-b> \rho.$
	\item [(ii)] If $r>s>1,$ then
	\[
	\depth M=\left\{
	\begin{array}{ll}
		s, & \text{ if } b-a\leq \sigma,\\
		r, & \text{ if } a-b\leq \rho \text{ and } a-b<-\sigma,
	\end{array}
	 \right.
	\] and $M$ is Cohen-Macaulay in all the other cases, that is, if  $-\sigma > a-b> \rho.$
	\item [(iii)] If $r>s=1,$ then $\depth M=1$ if $b-a\leq \sigma$ and $M$ is Cohen-Macaulay otherwise.
	\item [(iv)] If $r=s=1,$ then $M$ is Cohen-Macaulay for any $a,b.$
\end{itemize}
\end{Proposition}

\begin{proof} We first note that (iv) follows by \cite[Theorem 4.2.3]{GW}. Let us now prove statements (i)--(iii).

By \cite[Theorem 4.2.3]{GW}, $\dim R\sharp S=r+s-1.$
Let $\mm, \nn,$ and $\pp$ be the maximal graded ideals of $R,S,$ and, respectively, $R\sharp S.$
Since $\dim R>1$ or $\dim S>1, $ by \cite[Theorem 4.1.5]{GW}, we have
\begin{equation}\label{eq5}
H_\pp^q(M)\cong (R(a)\sharp H_\nn^q(S(b)))\dirsum (H_\mm^q(R(a))\sharp S(b))\dirsum (\dirsum_{i+j=q+1}H_\mm^i(R(a))\sharp H_\nn^j (S(b))),
\end{equation}
for every $q\geq 0.$ By local duality \cite[Theorem 3.6.19]{BHbook}, we have $H_\mm^q(R(\rho))\cong R^\vee$ if $q=r$ and
$H_\mm^q(R(\rho))=0$ if $q\neq r.$ Hence,
$
H_\mm^q(R(a))=H_\mm^q(R(\rho))(a-\rho)=R^\vee(a-\rho)$ if $q=r$ and $H_\mm^q(R(a))=0,$ otherwise. Similarly, we get
$H_\nn^q(S(b))=S^\vee (b-\sigma)$ if $q=s$ and $H_\nn^q(S(b))=0,$ otherwise.

(i). Let $r=s>1. $ By using (\ref{eq5}), we get $H_\pp^q(M)=0$ if $q\neq r, q<2r+1,$ and
\[
H_\pp^r(M)\cong (R(a)\sharp H_\nn^r(S(b)))\dirsum (H_\mm^r(R(a))\sharp S(b))\cong (R(a)\sharp S^\vee(b-\sigma))\dirsum(R^\vee(a-\rho)\sharp S(b)).
\]
Thus $H_\pp^r(M)\neq 0$ if and only if there exists $i$ such that $(R(a)\sharp S^\vee(b-\sigma))_i\neq 0$ or
$(R^\vee(a-\rho)\sharp S(b))_i\neq 0.$ This is equivalent to
$(a+i\geq 0$ and $b-\sigma +i\leq 0)$ or $(a-\rho+i\leq 0$ and $b+i\geq 0)$ which means $-a\leq i\leq \sigma-b$ or $-b\leq i\leq \rho-a.$
Therefore, $H_\pp^r(M)\neq 0$ if and only if $b-a\leq \sigma$ or $a-b\leq \rho.$ This implies (i).

(ii). Let $r>s>1.$ By using (\ref{eq5}), we get $H_\pp^s(M)\cong (R(a)\sharp H_\nn^s(S(b))).$  Thus, $H_\pp^s(M)\neq 0$ if and only if
$R(a)\sharp S^\vee(b-\sigma)_i\neq 0.$ With similar calculations as in case (i), we get $H_\pp^s(M)\neq 0$ if and only if $b-a\leq \sigma.$
Next, $H_\pp^r(M)\cong H_\mm^r(R(a))\sharp S(b)=R^\vee(a-\rho)\sharp S(b).$ Thus $H_\pp^r(M)\neq 0$ if and only if there exists $i$ with
$a-\rho+i\leq \sigma$ and $b+i\geq 0.$ Hence, $H_\pp^r(M)\neq 0$ if and only if $a-b\leq \rho $ and the conclusion of (ii) follows.

(iii). Let $r>s=1.$ By (\ref{eq5}), we get $H_\pp^1(M)\cong R(a)\sharp H_\nn^1 (S(b))\cong R(a)\sharp S^\vee(b-a).$ Thus,
$H_\pp^1(M)\neq 0$ if and only if there exists $i$ such that $i+a\geq 0$ and $b-\sigma+i\leq 0.$ This implies that $H_\pp^1(M)\neq 0$ if and only if $b-a \leq \sigma. $ In this case, $\depth M=1. $
\end{proof}

\begin{Corollary}\label{CM}
Let $(R,S)$ be a pair of friendly standard graded algebras and $T=R\sharp S.$ Assume that $R,S$ are Gorenstein of $a$-invariants $\rho,$ respectively, $\sigma$, and $\dim R,$ $ \dim S\geq 2.$ Then, the anticanonical module $\omega_T^\ast$ is Cohen-Macaulay if and only if $\sigma> 2\rho$ and $\rho> 2 \sigma$.
\end{Corollary}

\begin{proof}
By Corollary~\ref{Gor}, we have $\omega_T^\ast \cong R(-\rho)\sharp S(-\sigma).$ The conclusion follows by applying Proposition~\ref{Depth} for $a=-\rho$ and $b=-\sigma.$
\end{proof}

We can generalize the above corollary to any friendly family $R_1,\ldots,R_m$ of Gorenstein standard graded algebras of dimension at least two.

\begin{Theorem}
\label{mfactors-2}
Let $R_1,\ldots,R_m$ be a friendly family of standard graded algebras with $\dim R_i=d_i\geq 2$ for $1\leq i\leq m$.
Assume that $R_i$ is Gorenstein of $a$--invariant $-\rho_i$ for $1\leq i\leq m,$ with $\rho_1 \geq \rho_2 \geq \cdots \geq\rho_m$.
Set $T=R_1\sharp \cdots \sharp R_m$ and let $a$ be any integer.

Then the  $T$-module $M =\sharp_{i=1}^m R_i(-a\rho_i)$ is Cohen-Macaulay if and only if
\begin{eqnarray*}
(1-a) \rho_{\ell+1} &>& -a \rho_\ell ,\qquad \text{ for } \ell=1, \dots, m-1, \text{ if } a \leq 0, \text{ or} \\
a\rho_{\ell+1} &>& (a-1) \rho_\ell, \ \text{ for } \ell=1, \dots, m-1,  \text{ if } a>0.
\end{eqnarray*}
In particular, if $M$ is a Cohen-Macaulay module then $T$ is a Cohen-Macaulay ring.
\end{Theorem}

\begin{proof}
Let $M_i=R_i(-a\rho_i)$ and $\mm_i$ the graded maximal ideal of $R_i$ for $1\leq i\leq m.$  Let $\mm$ be the graded maximal ideal of $\sharp_{i=1}^m R_i.$
By applying induction on $m$ and \cite[Proposition~4.4.3, Theorem~4.1.5]{GW}, we obtain:
\begin{equation}\label{eq6bis}
 \dim \sharp_{i=1}^m M_i=\sum_{i=1}^m d_i-(m-1)
\end{equation}
	and
\begin{equation}\label{eq7bis}
H_{\mm}^q(M)\cong \Dirsum_{\ell =1}^m \Dirsum_{\stackrel{1\leq i_1<\cdots <i_{\ell}\leq m}{j_1+\cdots + j_\ell=q+(\ell-1)}}
M_1\sharp\cdots\sharp H_{\mm_{i_1}}^{j_1}(M_{i_1})\sharp \cdots \sharp H_{\mm_{i_\ell}}^{j_\ell}(M_{i_\ell})\sharp \cdots \sharp M_m.
\end{equation}

Using \cite[Theorems 3.5.7 and 3.6.18]{BHbook} we get
\begin{eqnarray}
\nonumber 	 H_{\mm_{i}}^{j}(M_i)&=&H_{\mm_{i}}^{j}(R_i(-a\rho_i))=0 \text{ if }j\neq d_i \text{ and } \\
\label{eq8bis}  H_{\mm_{i}}^{d_i}(M_i) &=& H_{\mm_{i}}^{d_i}(R_i(-a\rho_i)) \cong H_{\mm_i}^{d_i}(R_i)(-a\rho_i) \cong (\omega_{R_i}^\vee)(-a\rho_i) \\
\nonumber 	&\cong& (R_i(-\rho_i))^\vee(-a\rho_i) \cong R_i^\vee((1-a)\rho_i).
\end{eqnarray}

Set $d=\sum_{i=1}^m d_i-(m-1)$. We know that $M$ is Cohen-Macaulay if and only if $H_{\mm}^{q}(M)=0$ for all $q<d$.
When $q=d$, each summand in the RHS of  \eqref{eq7bis} corresponds to a choice
of $\ell$ with $1\leq \ell \leq m$, of $1 \leq i_1 < \dots < i_\ell \leq m$ and
$j_1,\dots, j_\ell$ such that $j_1+\dots+j_\ell = \sum_{i=1}^m (d_i-1) +\ell$.
The latter equation implies $\sum_{i=1}^\ell (j_i-1)=\sum_{i=1}^m(d_i-1)$.
Since $j_i \leq d_i$ we must have $\ell=m$, $j_i=d_i$ for $i=1,\dots, m$ and $i_1=1, \dots, i_m=m$, hence
$$
H_{\mm}^{d}(M) \cong \sharp_{i=1}^m H_{\mm_i}^{d_i}(M_i).
$$
On the other hand, by \eqref{eq8bis},
for any nonzero summand in the RHS of \eqref{eq7bis} with $\ell=m$ we must have $i_1=1, \dots, i_m=m$ and $j_1=d_1, \dots, j_m=d_m$,
therefore $q=\sum_{i=1}^m d_i- (m-1)=d$.

Therefore, $M$ is Cohen-Macaulay if and only if, for $1\leq \ell\leq m-1,$ and for all $1\leq i_1<\cdots< i_\ell \leq m, $ we have
\[
M_1\sharp\cdots\sharp H_{\mm_{i_1}}^{d_{i_1}}(M_{i_1})\sharp \cdots \sharp H_{\mm_{i_\ell}}^{d_{i_\ell}}(M_{i_\ell})\sharp \cdots \sharp M_m=0.
\]
Note that, by (\ref{eq8bis}),
\[
M_1\sharp\cdots\sharp H_{\mm_{i_1}}^{d_{i_1}}(M_{i_1})\sharp \cdots \sharp H_{\mm_{i_\ell}}^{d_{i_\ell}}(M_{i_\ell})\sharp \cdots \sharp M_m
\neq 0
\]
if and only if there exists $k$ such that $k-a\rho_i\geq 0$ for $i\not \in \{i_1,\ldots,i_\ell\}$ and $(1-a)\rho_i+k \leq 0$ for
$i\in \{i_1,\ldots,i_\ell\}.$
This is equivalent to the inequalities
\[
\min\{(a-1)\rho_i: i\in \{i_1,\ldots,i_\ell\}\}\geq k\geq \max\{a \rho_i:i\not \in \{i_1,\ldots,i_\ell\} \}.
\]

Thus, we derive that
\[
M_1\sharp\cdots\sharp H_{\mm_{i_1}}^{d_{i_1}}(M_{i_1})\sharp \cdots \sharp H_{\mm_{i_\ell}}^{d_{i_\ell}}(M_{i_\ell})\sharp \cdots \sharp M_m
= 0
\] if and only if
\begin{equation}\label{eq9bis}
\max\{a\rho_i: i\notin \{i_1,\ldots,i_\ell\}\} > \min\{(a-1)\rho_i:i \in \{i_1,\ldots,i_\ell\} \}.
\end{equation}

Consequently, $M$ is Cohen-Macaulay if and only if, for $1\leq \ell\leq m-1,$ and for all $1\leq i_1<\cdots< i_\ell \leq m, $ inequality
(\ref{eq9bis}) holds.

{\em Case $a<0$}. The inequalities \eqref{eq9bis} are equivalent to
\begin{equation}\label{eq9bis-a-negative}
(1-a) \max \{ \rho_i: i\in \{ i_1, \dots, i_\ell \}\} > -a \min \{ \rho_i: i\not\in\{i_1, \dots, i_\ell \}\}
\end{equation}
for all $\ell$ and $i_1, \dots, i_\ell$ as above.

Let $1\leq \ell\leq m-1. $ By (\ref{eq9bis-a-negative}), we have
\[
(1-a) \max\{\rho_{m-\ell+1},\ldots,\rho_m\} > -a \min\{\rho_1,\ldots,\rho_{m-\ell}\}
\]  which implies $(1-a)\rho_{m-\ell+1}> -a \rho_{m-\ell}$. Equivalently, $(1-a) \rho_{\ell+1} > -a \rho_\ell$  for $\ell=1, \dots, m-1$.

 Conversely, if $(1-a) \rho_{m-\ell+1}> -a\rho_{m-\ell}$ then, for any
$1\leq i_1<\cdots < i_\ell\leq m,$ we have
\begin{eqnarray*}
(1-a)\max\{\rho_i: i\in \{i_1,\ldots,i_\ell\}\} &\geq& (1-a) \rho_{m-\ell+1} \\
&>& -a \rho_{m-\ell} \geq -a\min\{\rho_i:i\not \in \{i_1,\ldots,i_\ell\} \},
\end{eqnarray*}
and \eqref{eq9bis-a-negative} holds.

{\em Case $a=0$}. According to \eqref{eq9bis}, $M$ is Cohen-Macaulay if and only if
$$
0>-\max\{ \rho_i: i\in \{i_1, \dots, i_\ell\}\}= -\rho_{i_1},
$$
for  all $\ell<m$ and $i_1<\dots<i_\ell$, which is equivalent to requiring $\rho_i>0$ for $i=1, \ldots, m$.

In {\em case $a=1$},  the inequalities \eqref{eq9bis} imply that $M$ is Cohen-Macaulay if and only if
   $\rho_i>0$ for $i=1, \dots, m$.

{\em Case $a>1$}. Now \eqref{eq9bis} implies that $M$ is Cohen-Macaulay if and only if
\begin{equation}
\label{eq9bis-a-positive}
a \max\{ \rho_i: i \not\in \{i_1, \dots, i_\ell\} \} > (a-1) \min\{\rho_i: i\in \{i_1, \dots, i_\ell \}\}
\end{equation}
for all $1\leq \ell < m$ and $1\leq i_1<\dots < i_\ell \leq m$. In particular, for any $1\leq \ell<m$ one has
$$
a \max\{\rho_{\ell+1}, \dots, \rho_m\} >(a-1) \min\{ \rho_1, \dots, \rho_\ell\},
$$
i.e. $a \rho_{\ell+1} > (a-1)\rho_\ell$ for $1\leq \ell <m$. We check that the latter  inequalities imply \eqref{eq9bis-a-positive}.
Indeed, for $\ell<m$ and $i_1<\dots<i_\ell$ we have
\begin{eqnarray*}
a \max\{\rho_i: i\not\in\{i_1, \dots, i_\ell\} \} &\geq& a \{ \rho_{\ell+1}, \dots, \rho_m\} =a \rho_{\ell+1} \\
&>& (a-1)\rho_\ell= (a-1)\min\{\rho_1, \dots, \rho_\ell\} \\
&\geq& (a-1)\min\{ \rho_i: i\in \{i_1, \dots, i_\ell\}\}.
\end{eqnarray*}

Let us assume that $M$ is a Cohen-Macaulay $T$-module. To prove that $T$ is a Cohen-Macaulay ring is equivalent
to the a-invariants of $R_i$ be negative for all $i$, see \cite[Theorem 4.4.1(i)]{GW}, or equivalently that $\rho_m>0$.
This was checked above for $0 \leq a\leq 1$.

If $a>1$,   the inequality $a \rho_m> (a-1)\rho_{m-1}$ gives $\rho_m> (a-1)(\rho_{m-1}-\rho_m) \geq 0$.
Similarly, when $a<0,$  the inequality $(1-a)\rho_m >a\rho_{m-1}$ implies that $\rho_m > a(\rho_m-\rho_{m-1})\geq 0$, as desired.
\end{proof}

\begin{Remark}
{\em
 In the setting of Theorem \ref{mfactors-2}, for $a=0$ we get  $M=\sharp_{i=1}^m R_i=T$,
and if $a=1$ we get  $M=\sharp_{i=1}^m R_i(-\rho_i)\cong \sharp_{i=1}^m\omega_{R_i}\cong \omega_T$, see \cite[Theorem 4.3.1]{GW}.
}
\end{Remark}

\begin{Corollary}
\label{cor:cm-chain}
In the setting of Theorem \ref{mfactors-2}, assume $a\neq 0,1$. Denote $C=(a/(a-1))^{sgn(a)}$. Then $M$ is a Cohen-Macaulay $T$-module if and only if
$$
C^{m-1}\rho_m> C^{m-2}\rho_{m-1}>\cdots >C\rho_2 >\rho_1.
$$
In particular, the anticanonical module $\omega_T^\ast$ is Cohen-Macaulay if and only if
\[2^{m-1}\rho_m>2^{m-2}\rho_{m-1}>\cdots> 2\rho_2>\rho_1.\]
\end{Corollary}

\begin{proof}
The first part follows immediately from Theorem \ref{mfactors-2}.  For the computation of the anticanonical module $\omega_T^\ast$ we use
that $\omega_{R_i}=R_i(-\rho_i)$, hence $\omega_{R_i}^\ast =R_i(\rho_i)$ for all $i$.  Moreover,  since the algebras $R_1, \dots, R_m$ are friendly, \cite[Theorem 4.3.1]{GW}
yields
\[
\omega_T^\ast\cong (\sharp_{i=1}^m\omega_{R_i})^\ast \cong \sharp_{i=1}^m\omega_{R_i}^\ast  \cong \sharp_{i=1}^m R_i(\rho_i).
\]
\end{proof}

\begin{Proposition}
\label{prop:cm-twists}
Let $R_1,\ldots,R_m$ be a friendly family of standard graded algebras with $\dim R_i\geq 2$ for $1\leq i\leq m$.
Assume that $R_i$ is Gorenstein of $a$--invariant $-\rho_i$ for $1\leq i\leq m$  with $\rho_1 \geq \rho_2 \geq \cdots \geq\rho_m>0$.
Let $\rho=\max \{\rho_i/\rho_{i+1}: i=1, \dots, m-1\}$. For any integer $a$ we set
$T^{\{a\}}= \sharp_{i=1}^m R_i(-a\rho_i)$ and we let $T=T^{\{0\}}$.

Then the  $T$-module $T^{\{a\}}$ is Cohen-Macaulay if and only if either $\rho=1$ or
$$
\frac{1}{1-\rho} <a <\frac{\rho}{\rho-1}.
$$
\end{Proposition}

\begin{proof}
Since $\rho_i >0$ for all $i,$ it follows that $T$ is a Cohen-Macaulay ring,  cf. \cite[Theorem 4.4.4(i)]{GW}.
Our hypothesis on the $\rho_i$'s implies that $\rho \geq 1$. Equality holds  if and only if $\rho_1=\dots=\rho_m$, i.e. $T^{\{a\}}$ is a shifted copy of $T$, hence Cohen-Macaulay.
As a matter of facts, the case $\rho=1$ is equivalent to $T$ being Gorenstein, by \cite[Theorem 4.4.7]{GW}.

Assume $\rho>1$.	Then $1/(1-\rho) <0$ and $1<\rho/(\rho-1)$. Since $T$ is Cohen-Macaulay it follows from Theorem \ref{mfactors-2} that  $T^{\{1\}}$ is Cohen-Macaulay, too.

If $a>1$, by Theorem \ref{mfactors-2} we see that $T^{\{a\}}$ is Cohen-Macaulay if and only if $a/(a-1) >\rho_\ell /\rho_{\ell+1}$
for $\ell=1, \dots, m-1$, i.e. $a/(a-1)>\rho$, which is equivalent to $1+1/(a-1)>\rho$, and to $1/(a-1)>\rho-1$. Thus $a<\rho/(\rho-1)$.

If $a<0$, by Theorem \ref{mfactors-2} it follows that $T^{\{a\}}$ is Cohen-Macaulay if and only if $(a-1)/a > \rho$, which is  equivalent to $a>1/(1-\rho)$.
\end{proof}

\begin{Corollary}
\label{cor:notwist}
In the setting of Proposition \ref{prop:cm-twists}, assume $\rho>1$. Then $T^{\{a\}}$ is Cohen-Macaulay only for $a=0,1$ if and only if $\rho \geq 2$.
In particular, the anticanonical module $\omega_T^\ast$ is not Cohen-Macaulay when $\rho \geq 2$.
\end{Corollary}

\begin{proof}  It is clear that $1/(1-\rho)<0$ and $\rho/(\rho-1)=1+1/(\rho-1) >1$, hence  $T^{\{a\}}$ is Cohen-Macaulay for $a=0,1$.
That these are the only values, by  Proposition \ref{prop:cm-twists},  it is equivalent to having $-1\leq 1/(1-\rho)$ and $\rho/(\rho-1) \leq 2$.
It is easy to check that both conditions mean $\rho \geq 2$.

As noted in the proof of Corollary \ref{cor:cm-chain}, the anticanonical module $\omega_T^\ast \cong T^{\{-1\}}$ and by the above arguments it is not Cohen-Macaulay when $\rho \geq 2$.
\end{proof}

By \cite[Proposition 3.3.18]{BHbook}, for a standard graded Cohen-Macaulay algebra $R$ which is  generically Gorenstein (for instance, a domain),
its canonical module may be identified with an ideal in $R$, that we call a canonical ideal of $R$. If $I$ and $J$ are two canonical
ideals of $R$, they are isomorphic and there exists an element $x$ invertible  in $Q(R)$ such that $I=xJ$.
Hence, for any integer $a$ the $T$-modules $I^a$ and $J^a$ are isomorphic, and Cohen-Macaulay or not at the same time.

The next result describes which powers of the canonical ideal of a Segre product of some friendly Gorenstein algebras are Cohen-Macaulay.

\begin{Proposition}
\label{prop:powers-cm}
Let $R_1, \dots, R_m$ be a friendly family of standard graded Gorenstein $K$-algebras of dimension at least two
such that $R_1\otimes_K \cdots \otimes_K R_m$ is a domain.
Let $-\rho_i$ be the $a$--invariant of $R_i$ for $i=1, \dots, m$. Assume $\rho_1\geq \rho_2 \geq \cdots \geq \rho_m>0$ and that
$\rho= \max \{   \rho_i/ \rho_{i+1}:i=1, \dots, m -1 \} >1$.

Let $T=\sharp_{i=1}^m R_i$, $\omega_T$   a canonical ideal of $T$, and $a$ any integer. Then $\omega_T^a$ is a Cohen-Macaulay $T$-module
if and only if
$$
\frac{1}{1-\rho} <a <\frac{\rho}{\rho-1}.
$$
\end{Proposition}

\begin{proof} The Segre product $T$ is a Cohen-Macaulay ring which is not Gorenstein,
since all $\rho_i$'s are positive and $\rho>1.$
By  \cite[Theorem 4.3.1]{GW}, the canonical module of $T$ is isomorphic to $\sharp_{i=1}^m R_i(-\rho_i)$, which
 is generated by $1\otimes R_{2,\rho_1-\rho_2}\otimes\cdots \otimes R_{m, \rho_1-\rho_m}$.
Here $R_{i,j}$ denotes the $j$th homogeneous component of $R_i$ for $i=1, \dots, m$ and $j\geq 0$.
Pick   $f=u_1\otimes \cdots \otimes u_{m-1} \otimes 1$ with $0\neq u_i \in R_{i,\rho_i-\rho_m}$ for $i=1, \dots, m-1$.
Since $\otimes_{i=1}^m R_i$ is a domain it follows that $T$ is a domain, as well. Therefore, the map
$\phi: \sharp_{i=1}^m R_i(-\rho_i) \stackrel{\cdot f}{\rightarrow} T$ is injective.
Without loss of generality we may assume that  $\omega_T$ is the image of $\phi$.
Hence
$$
\omega_T =(u_1\otimes u_2 R_{2,\rho_1-\rho_2}\otimes\cdots \otimes u_{m-1}R_{m-1,\rho_1-\rho_{m-1}} \otimes R_{m,\rho_1-\rho_m})T.
$$

Assume $a>0$. Since the algebras $R_i$ are standard graded, we get
$$
\omega_T^a=(u_1^a\otimes u_2^a R_{2,a(\rho_1-\rho_2)}\otimes\cdots \otimes u_{m-1}^a R_{m-1,a(\rho_1-\rho_{m-1})} \otimes R_{m,a(\rho_1-\rho_m)})T.
$$
Arguing as above, we have that $\omega_T^a$ is isomorphic to the $T$-module $\sharp_{i=1}^{m}R_i(-a\rho_i)$, for all $a>0$.

With a similar argument we may identify  the anticanonical module $\omega_T^\ast \cong \sharp_{i=1}^m R_i(\rho_i)$  with the ideal
$$
\omega_T^\ast =(R_{1,\rho_1-\rho_m}\otimes v_2 R_{2,\rho_2-\rho_m}\otimes\cdots \otimes v_{m-1}R_{m-1,\rho_1-\rho_{m-1}} \otimes v_m )T,
$$
where $0\neq v_i\in R_{i,\rho_1-\rho_i}$ for   $i=2,\dots, m$.
A similar discussion shows the isomorphism $\omega_T^a \cong \sharp_{i=1}^{m}R_i(-a\rho_i)$  for negative $a$, as well.

The conclusion now follows by Theorem \ref{prop:cm-twists}.
\end{proof}

\medskip
{\bf Acknowledgement}.
 We thank an anonymous referee for valuable comments, in particular for suggesting  a correction to Theorem \ref{ToricSegre}. 

Dumitru Stamate  was supported by a fellowship at the Research Institute of the University of Bucharest (ICUB).

\end{document}